\documentclass{article}
\title{Some remarks on the Allen-Cahn equation in $\mathbb{R}^n$}
\usepackage{graphicx} 

\usepackage{fancyhdr}
\usepackage{amsthm}
\date{}

\usepackage[titles]{tocloft}

\usepackage{hyperref}
\usepackage[utf8]{inputenc}
\usepackage[T1]{fontenc}
\usepackage{subfig}
\usepackage{pgfplots}
\usepackage{tikz}
\usepackage{braket}
\usepackage{geometry}
\usepackage{setspace}
\usepackage{enumitem}
\usepackage{float}
\usepackage{amsmath,amssymb,latexsym,esint,cite,mathrsfs,lipsum}
\usepackage{amsthm}
\usepackage{amsfonts}
\usepackage{bm}
\usepackage{bbm}
\usepackage[english]{babel}
\usepackage[nouppercase]{frontespizio}
\theoremstyle{plain}
\newtheorem{definition}{Definition}[section]
\newtheorem{proposition}[definition]{Proposition}
\newtheorem{theorem}{Theorem}

\newtheorem{conjecture}[theorem]{Conjecture}

\newtheorem{remark}[theorem]{Remark}

\numberwithin{theorem}{section}

\author{Gabriele Ferla}
\date{}

\begin{document}
\maketitle
\begin{abstract}
In this short note we present new results on a higher-dimensional generalization of De~Giorgi's conjecture for Allen--Cahn type equations, focusing on dimensions $n \ge 9$. 
Although counterexamples are known in this regime, our goal is to identify assumptions on solutions that still enforce one-dimensional symmetry.

Our approach is inspired by recent advances on Bernstein's problem, in particular by~\cite{f}. 
We prove an analogue of Savin's theorem (see~\cite{s}) in arbitrary dimension: for energy-minimizing solutions whose level sets enjoy $n-7$ directions of monotonicity, we deduce one-dimensional symmetry. 
The result is sharp, since the example constructed in~\cite{dpk} exhibits only $n-8$ monotonicity directions.

In the same spirit, we extend these ideas to nonlocal phase transitions, following the corresponding result in~\cite{FCV}, and we discuss an application to free boundary problems as in~\cite{fb}. 
Finally, we establish a counterpart of the Ambrosio--Cabr\'e theorem~\cite{ac} for solutions that are not necessarily energy minimizers and may lack bounded energy density, assuming instead $n-2$ directions of monotonicity everywhere.

Overall, this note aims to further strengthen the connection between phase transition models and minimal surface theory.
\end{abstract}
\section{Introduction and main results}
We consider, as in~\cite{FV}, a bounded solution of the general equation,
\begin{equation}\label{dwplap}
\Delta_p u = W'(u),
\end{equation}
which, according to the discussion in~\cite{FV}, satisfies $|u|\leq 1$.\\
We recall that $\Delta_p u = \operatorname{div}(|\nabla u|^{p-2}\nabla u)$, following the notation in~\cite{FV}.\\
In this case, the associated energy functional is
\begin{equation}\label{energyplap}
E(u,\Omega)=\int_{\Omega}\frac{|\nabla u|^p}{p}+W(u)\,dx.
\end{equation}
Without loss of generality, we consider the following definition taken from~\cite{FV}.
\begin{definition}\label{dw}
    We take $W: \mathbb{R}\to \mathbb{R}$ to be a double-well potential. More precisely, we suppose that:\\
1) $ 
W \in C^{1,\alpha}_{\mathrm{loc}}(\mathbb{R})\cap C^{1,1}_{loc}(-1,1)$ 
for some $\alpha \in (0,1)$,\\
2) $W(r) > 0$ for any $r \in \mathbb{R} \setminus \{-1,+1\}$\\
3) $W(-1) = W(+1) = 0.$\\
We also suppose that $W'(r) = 0$ if and only if $r \in \{-1, \kappa, 1\}$ where $\kappa \in (-1,1)$.  \\
We fix $p \in (1,+\infty)$.\\
Moreover, we take the following growth conditions near the two wells of $W$. We suppose that there exist some $0 < c < 1 < C$ and some $\theta^* \in (0,1)$ such that:
\begin{itemize}
  \item For any $\theta \in [0,1]$,
  \[
  c \, \theta^p \;\leq\; W(-1+\theta) \;\leq\; C \theta^p
  \qquad \text{and} \qquad
  c \, \theta^p \;\leq\; W(1-\theta) \;\leq\; C \theta^p.
  \]

  \item For any $\theta \in [0,\theta^*]$,
  \[
  c \, \theta^{p-1} \;\leq\; W'(-1+\theta) \;\leq\; C \theta^{p-1}
  \qquad \text{and} \qquad
  -C \theta^{p-1} \;\leq\; W'(1-\theta) \;\leq\; -c \theta^{p-1}.
  \]

  \item $W'$ is monotone increasing in $(-1,-1+\theta^*) \cup (1-\theta^*,1)$.
  \end{itemize}
\end{definition}
The classical potential $W(u)=1/4(1-u^2)^2$, associated with the classical  Allen-Cahn equation $$\Delta u=u^3-u$$, corresponds to the case $p=2$.\\
Under this assumption, using the classical elliptic estimates, $u \in C^{1,\beta}(\mathbb{R}^n)$ for some $\beta>0$.
In this setting, one may adopt the more general definition below, allowing for different potentials, provided stronger regularity assumptions are imposed.\\
This formulation is taken from~\cite{w}.
\begin{definition}[General definition with more regularity]\label{DW}
We take $W :\mathbb{R}\to\mathbb{R} \in C^2_{loc}(\mathbb{R})$  such that:\\
\textbf{1)} $W(1)=W(-1)=0$ , $W>0 $ in $(-1;1)$ ,\\
\textbf{2)} There exist a $\gamma \in (-1,1)$ and a $k >0$ with $W'>0 $ in $(-1;-\gamma)$ and $W'<0 $ in $(\gamma; 1)$ ;
\textbf{3)}$W''(s) \geq k$ for $|s| \geq \gamma$
\end{definition}
A fundamental problem concerning PDEs of the type $\Delta u=W'(u)$ is De Giorgi's conjecture, formulated in~\cite{edg}.
\begin{conjecture}
    Let $u \in C^2(\mathbb{R}^n)$ be a bounded and entire solution of 
    \begin{equation}\label{dwpde}
    \Delta u=W'(u)
    \end{equation}
    where $W$ is a double well potential as in definition \ref{DW} and assume that $\partial_nu>0$.\\
    Then the level sets $\{u=k\}$ are hyperplanes, provided $n \leq 8$.
\end{conjecture}
This conjecture has been solved for $n \leq 3$ by Ambrosio--Cabr\'e in~\cite{ac}.\\
Under the additional assumption
\[
\lim_{x_n \to \pm \infty} u(x',x_n)=\pm 1
\quad \text{for all } x' \in \mathbb{R}^{n-1},
\]
the so-called weak form of the conjecture has been proved by Savin~\cite{s} and Wang~\cite{w}, using different approaches.
\begin{remark}
It is worth noting that, in the case $p=2$ of definition \ref{dw}, it is enough to require that there exists a $x' \in \mathbb{R}^{n-1}$ such that $\lim_{x_n \to \pm \infty} u(x',x_n)=\pm 1$, according to Lemma~2.1 of ~\cite{FV}. 
\end{remark}
It is well known that this assumption implies that the solution is a local energy minimizer for the energy functional 
\begin{equation}\label{energiafunzionale}
E(u,\Omega)=\int_{\Omega} \frac{1}{2}|\nabla u|^2+W(u)dx
\end{equation}\\
In higher dimensions, several conditions ensuring one-dimensional symmetry of solutions have been established (see, for instance,~\cite{FV,s,w}), inspired by the classical Bernstein-type theorems presented in the celebrated book by Giusti (see~\cite{g}).
Here, we aim to provide additional simple conditions, in the spirit of the more recent Bernstein-type theorem stated in~\cite{f}.\\
For the reader’s convenience, we recall its statement.
\begin{theorem}\label{farinabernstein1}
       If $\varphi: \mathbb{R}^n \to \mathbb{R}$ is a $C^2$ function with $\partial_i\varphi >k$ for some $k \in \mathbb{R}$ and for all $i \in [8,\ldots, n]$, and if $\text{graph}(\varphi)$ is a minimal graph, then $\varphi$ is affine.
\end{theorem}
Inspecting the proof, one observes that the method can be applied directly to the Allen–Cahn equation (and, more generally, to equations with a double-well potential).
We therefore state a rigidity theorem for Allen--Cahn type PDEs, also taking into account the $p$-Laplace operator, since the flatness theorem of Savin remains valid in this setting (see~\cite{fvplap,sp}).\\
As mentioned above, a flatness theorem can be established in this framework.\\
Let $u$ be a solution of \eqref{dwplap} minimizer of the functional (\ref{energyplap}).
If 
\begin{equation}\label{scaled}
u_\varepsilon(x)=u(x/\varepsilon)
\end{equation}
then
$\varepsilon^p\,\Delta_p u_\varepsilon = W'(u_\varepsilon)$ and $u_\varepsilon$ are energy minimizers of the scaled functionals
\begin{equation}\label{scaledfunctionals}
    E^\varepsilon(u,\Omega)=\int_{\Omega}\frac{\varepsilon^{p-1}|\nabla u_\varepsilon|^p}{p}+\frac{1}{\varepsilon}W(u_\varepsilon)\,dx.
\end{equation}
Energy-minimizing (or suitably controlled critical) solutions of the $\varepsilon$--Allen--Cahn equation converge (along subsequences) in $L^1_{\mathrm{loc}}$ (see~\cite{sp,fvplap}) to a $\pm1$-valued function; the diffuse interfaces converge to a set of minimal perimeter i.e., $\partial E$ is a minimal hypersurface.\\
It follows from~\cite{sp,fvplap} that the scaled level sets $\{u_\varepsilon=0\}$ converge (up to a subsequence) to $\partial E$.\\
If $E$ is a half space, then $u$ coincides with the one-dimensional heteroclinic solution (see, for instance, Corollary~7 in~\cite{fvplap}).\\
We are now ready to state the main theorem.
\begin{theorem} \label{n-7derivatives}
     Let $u \in C^{2}(\mathbb{R}^{n})$ be a local minimizer for the energy functional (\ref{energyplap}) and an entire, bounded solution of (\ref{dwplap})
    where $W$ is as in definition \ref{DW} with $p=2$ , or as in definition \ref{dw}.\\
    Assume that $\partial_i u >0 \quad \forall i \in [8,...n]$ on $\{u=0\}$ (equivalently, on a level set $\{u=t\}$ with $|t| < 1$).\\
    Then $u$ is the one dimensional heteroclinic solution.
\end{theorem}
\begin{remark}
In this theorem, it is sufficient to consider an energy-minimizing solution 
$u$ such that one of its level sets $\{u=t\}$ can be represented by a function $\varphi: \mathbb{R}^{n-1} \to \mathbb{R}$ for which there exists $k \in \mathbb{R}$ such that $\partial_i \varphi >k$ for all $i \in [8,...n-1]$.
\end{remark}
The previous theorem generalizes Theorems~2.3 and~2.4 contained in~\cite{s}, and for $n=8$ it corresponds to the weak form of the conjecture.\\
A simple observation is that the same result can be stated for quasi-minimizers.\\
This follows directly from~\cite{fvplap}.\\
We prefer to keep the present formulation in order to provide a clear discussion.\\
Another natural related question is to identify conditions weaker than energy minimality under which a solution is one-dimensionally symmetric.
Motivated by the main theorem of Ambrosio--Cabr\'e (see~\cite{ac}, Theorem~1.2), we prove one-dimensional symmetry for solutions that are monotone in $n-2$ linearly independent directions everywhere.\\
In the results that follow, we consider general monotone solutions, without assuming that they are energy minimizers.\\
We first state the theorem in its most general form, and then derive some simpler and more transparent consequences. In this setting, we do not assume either bounded energy density or energy minimality.\\
Moreover, we show that the theorem remains valid under weaker regularity assumptions on the double-well potential, as specified in definition~\ref{dw}.
We begin with a definition that simplifies the notation.
\begin{definition}
    We say that a set $C \subseteq S^{n-1}$ is a $n-2$ cone around $e_n$ if there exists $\varepsilon >0$ such that:
    $$C=\{ v \in S^{n-1} \cap \text{span}(e_3,..e_n): |v - e_n|\leq \varepsilon\}$$
\end{definition}
We can now state the main results.
\begin{theorem}\label{n-2}
     Let $u \in C^{2}(\mathbb{R}^{n})$ be a bounded entire solution of (\ref{dwpde}) where $W \in C^{1,1}_{loc}$ is a double well potential as in definition \ref{dw}, with $p=2$.\\
    We denote by $\kappa$ the constant introduced in definition \ref{dw} .\\
    Assume that :\\
    \textbf{1)} $\partial_n u >0$ everywhere\\
    \textbf{2)} There exists a $\gamma >0$ such that for each $\kappa - \gamma<\delta< \kappa+ \gamma$ there exists a $n-2$ dimensional cone $C(\delta)$ such that:
    $$\partial_{v}u(x) \geq 0 \quad  \text{for every $x \in \{u= \delta \}$ and every $v \in C(\delta)$}.$$\\
    Then the solution $u$ is the one dimensional heteroclinic solution.
\end{theorem}
We now state a variant of the previous theorem, which is also valid for a different class of operators.
    \begin{proposition}\label{n-2dirp}
        Let $u \in C^{1}(\mathbb{R}^{n})$) be a bounded solution of (\ref{dwplap}) where $W$ is a double well potential as in definition \ref{dw} with $p \in (1,+\infty)$, or as in definition \ref{DW} with $p=2$.\\
    Assume that :\\
    \textbf{1)} $\partial_n u >0$ everywhere\\
    \textbf{2)} For each $\delta$ with $|\delta|< 1$ there exists there exists a $n-2$ dimensional cone $C(\delta)$ such that:
    $$\partial_{v}u(x) \geq 0 \quad  \text{for every $x \in \{u= \delta \}$ and every $v \in C(\delta)$}.$$\\
    Then the solution $u$ is the one dimensional heteroclinic solution.
    \end{proposition}
In this spirit, we are also able to prove a result that provides a condition under which a solution of (\ref{dwpde}) is energy minimizing.\\
In the following proposition, we denote by $B'_r$ the ball of radius $r>0$ in $\mathbb{R}^{n-1}$ centered in the origin.
 \begin{proposition}\label{n-2dir}
        Let $u \in C^{1}(\mathbb{R}^{n})$be a bounded solution of (\ref{dwpde}) 
        where $W$ is a double-well potential as in definition \ref{DW}, or as in definition \ref{dw}.\\
        Assume that:\\
    \textbf{1)} $\partial_n u >0$ everywhere\\
    \textbf{2)} There exists a $n-2$ cone $C$ and an $r>0$ such that, for every $x$ in the infinite cylinder $K=B'_r \times \mathbb{R}$ we have \\
 $$\partial_{v}u(x) \geq 0 \quad  \text{for every $v \in C$}.$$\\
    Then the solution $u$ is an energy minimizing solution.
    \end{proposition}
The note is organized as follows.\\
The subsequent sections are devoted to the proof of Theorem~\ref{n-7derivatives} and of the results concerning monotonicity in $n-2$ directions.\\
In the case of Theorem~\ref{n-7derivatives}, we straightforwardly extend the argument to the setting of non-local phase transitions and to free-boundary problems.\\
In the proof of Theorem~\ref{n-2} and its related variants, we make essential use of several results from~\cite{FV} on limit interfaces, combined with Theorem~\ref{n-7derivatives}. We also discuss some possible extensions of this theorem that we were not able to prove.\\
We conclude the note by establishing an asymptotic property of solutions to (\ref{dwpde}) in $\mathbb{R}^n$.

\section{Proof of Theorem \ref{n-7derivatives} and some applications}
The aim of this section is to prove Theorem~\ref{n-7derivatives} along with some related results.

Recall that for a solution $u$ of an Allen--Cahn type equation, we define the scaled functions as in (\ref{scaled}).

By the classical theorem of Modica and its variants (see~\cite{s,sp}), if $u$ is an energy-minimizing solution of \eqref{dwplap} for the functional \eqref{energyplap}, then, up to a subsequence,
\[
u_\varepsilon \to \chi_E - \chi_{E^c} \quad \text{in } L^1_{\rm loc}(\mathbb{R}^n),
\]
where $E \subset \mathbb{R}^n$ is a minimal perimeter set, and, moreover, each scaled level set $\{u_\varepsilon = t\}$ converges, in the sense of the Hausdorff distance, to $\partial E$.\\
We now state a general version of the flatness theorem contained in~\cite{s,sp} (for a useful reference, see Corollary~7 of~\cite{fvplap}).

\begin{theorem}\label{Savin}
Let $u:\mathbb{R}^n \to \mathbb{R}$ be a bounded entire solution of \eqref{dwplap}, which is energy-minimizing for the functional \eqref{energyplap}.
Let $E \subset \mathbb{R}^n$ be the minimal perimeter set such that
\[
u_\varepsilon \to \chi_E - \chi_{E^c}.
\]
If $E$ is a half-space, then $u$ is a one-dimensional heteroclinic solution.
\end{theorem}
We are now ready to prove Theorem~\ref{n-7derivatives}, following ideas similar to those used in the proof of Theorem~\ref{farinabernstein1}.\\
Before proving the theorem we define
\begin{equation}\label{insieme positivo}
E' (\psi):= \{ x' \in \mathbb{R}^{n-1} : \psi(x') \in \mathbb{R} \}.
\end{equation}
Where $\psi : \mathbb{R}^{n-1} \to \overline{\mathbb{R}}$

\begin{proof}[Proof of Theorem~\ref{n-7derivatives}]
It is sufficient to closely follow the proof of the main theorem in~\cite{f}.

It is standard to verify that the level set $\{u=0\}$ can be represented as the graph of a function
\[
\varphi : \mathbb{R}^{n-1} \to \overline{\mathbb{R}}, \quad \partial_i\varphi(x') >k \quad \text{for all} \quad x' \in E'(\varphi) \quad \text{for all} \quad i \in \{8,\dots n-1\}
\]
with $\varphi \in C^{1}(E'(\varphi))$, and we denote by $S$ its subgraph.

Without loss of generality we can suppose that $0 \in \{u=0\}$.\\
We consider the blow-down sequence $\varphi_\varepsilon(x') := \varepsilon \varphi(x'/\varepsilon)$.
For each $\varphi_\varepsilon$ we define the corresponding subgraph $S_\varepsilon$.
It is standard to verify that for each $\varepsilon >0$ we have that $$\partial_i \varphi_\varepsilon(x') > k \quad \text{for all} \quad x' \in E'(\varphi_\varepsilon) \quad \text{,for all} \quad i \in \{8,\dots n\}$$

Using the local monotonicity assumption along $e_i$ for all $i \in {8,...n}$, it's natural to verify that for each $S_\varepsilon $ we have that $S_\varepsilon - e_i \subseteq S_\varepsilon$.

By Modica’s convergence theorem and its variants contained in ~\cite{sp,fvplap} , the sets $S_\varepsilon$ converge, up to a subsequence, in $L^1_{\mathrm{loc}}$, to a minimal perimeter set $E$ such that $E- e_i \subseteq E$  for all $i \in {8,...n}$.

Performing a second blow down on the minimal perimeter set $E$ we obtain the sequence $E_\varepsilon$, and, following the standard procedure as in~\cite{g} , we obtain a minimal cone $E*$ such that $E_\varepsilon \to E*$ in $L^1_{\mathrm{loc}}$ and $E*- e_i \subseteq E*$  for all $i \in {8,...n}$.

Therefore, Proposition~2.1 in~\cite{f} applies and yields that $E*$ is a cylinder in the $e_i$ direction for all $ i \in \{8 \dots n\} $.

Hence,
\[
E* = M \times \mathbb{R}^{n-7},
\]
where $M$ is a minimal perimeter set in $\mathbb{R}^7$.
By the  classification of minimal cones in dimension $7$ in~\cite{g} , $M$ must be a half-space.
Therefore, $E*$ and $E$ itself is a half-space.

Finally, we apply Savin’s theorem, Theorem \ref{Savin} to conclude that $u$ is the one-dimensional heteroclinic solution.
\end{proof}

The previous result can be stated also in the non local setting, with some small differences in the statement.
\begin{definition}[Notation about the problem in the non local setting]
  The nonlocal Allen--Cahn energy functional in the fractional setting is given by
\begin{equation} \label{fractionalminimizer}
\mathcal{E}_s(u)
=
\frac{1}{4}
\iint_{\mathbb{R}^{2n}}
\frac{|u(x)-u(y)|^2}{|x-y|^{n+2s}}\,dx\,dy
+
\int_{\mathbb{R}^n} W(u(x))\,dx,
\end{equation}
where $s\in(0,1)$, and $W$ is a double-well potential as in Definition~\ref{dw}.\\
We can directly define the scaled functional
\[
\mathcal{E}^\varepsilon_s(u)
=
\frac{1}{4}
\iint_{\mathbb{R}^{2n}}
\frac{|u(x)-u(y)|^2}{|x-y|^{n+2s}}\,dx\,dy
+
\varepsilon^{-s}\int_{\mathbb{R}^n} W(u(x))\,dx,
\]
It is associated with solutions of the PDE 

\begin{equation}\label{fractionalpde}
(-\Delta u)^s+W'(u)=0
\end{equation}
\end{definition}
The correspondent of the Modica Theorem can be proved in this setting, as shown in Theorem~1.3 of ~\cite{SavinValdinoci2012}.\\
It is known, from ~\cite{DipierroSerraValdinoci2016,s1,s2},  that the analogue of Savin's flatness theorem remains valid in the fractional framework and that the blow-down limit of the scaled solutions is a $s$-minimal cone.
In particular, we can state this general result.
\begin{proposition}
For every $j \leq 7$ there exists $s(j) \leq \frac{1}{2}$ such that ,
if $u \in C^1(\mathbb{R}^n), |u| \leq 1$ is an energy minimizer of \eqref{fractionalminimizer} and a solution of \eqref{fractionalpde} with $s(j)\leq s \leq 1$ ,whose level set $\{u=0\}$ enjoys $n-j$ directions of monotonicity , then $u$ is the one-dimensional heteroclinic solution.
\end{proposition}
\begin{proof}
    As in the previous proof, we suppose without loss of generality that $0 \in \{u=0\}$ and consider the blow-down sequence $u_\varepsilon$, defined as in \eqref{scaled}, which, according to the results in ~\cite{SavinValdinoci2012,DipierroSerraValdinoci2016}, converges in $L^1_{loc}$ to $\chi_E-\chi_{E^c}$ where $E$ is a $2s$ minimal cone.\\
    Following line by line the previous proof, we have that for every subgraph $S_\varepsilon$ of the blow-down sequence, $S_\varepsilon - e_i \subseteq S_\varepsilon$ for all $i \in {j+1 \dots n}$ and therefore the limit $2s$-minimal cone $E$ enjoys this property.\\
    We can now apply Proposition~3.1 in~\cite{FCV} (which is the analogue of Proposition~2.1 in~\cite{f}), proving that the set $E$ is a cylinder along every $e_i$ for all $i \in {j+1 \dots n}$ and therefore it takes the form:
    $$E=M \times \mathbb{R}^{n-j}$$
    where $M$ is a $2s$ minimal perimeter set in $\mathbb{R}^j$.\\
    Reasoning as in the previous proof and using the standard classification of $s$-minimal perimeter sets contained in ~\cite{FCV}, we prove that the set $E$ is an hyperplane. To conclude the proof it suffices to apply the improvement of flatness theorems in the different cases , respectively Theorem~1.2 of ~\cite{DipierroSerraValdinoci2016}, Theorem~1.1 of ~\cite{s1} and Theorem~1.1 of ~\cite{s2}.
\end{proof}

A simple variant of the previous results is an extension to free boundary problems, providing additional conditions in higher dimensions to establish one-dimensional symmetry for solutions of
\begin{equation}\label{serrin}
\begin{cases}
\Delta u = W'(u) & \text{in } \Omega, \\[2mm]
u > 0, \quad \partial_n u > 0 & \text{in } \Omega,\\[1mm]
u = 0 & \text{on } \partial \Omega,\\[1mm]
|\nabla u| = \text{const} & \text{on } \partial \Omega.
\end{cases}
\end{equation}

Problems of this type were raised in~\cite{fb}.

\begin{theorem}
Let $u$ be a solution of \eqref{serrin}, and suppose that $\Omega$ is the epigraph of a smooth function
\[
\varphi : \mathbb{R}^{\,n-1} \to \mathbb{R},
\]
with $\partial_i \varphi$ bounded from below (or above, without loss of generality) for all $i \in \{8,\dots,n\}$. 

Then $\varphi$ must be affine, and $u$ is the one-dimensional heteroclinic solution.
\end{theorem}

\begin{proof}
   We closely follow the proof of Theorem 2.1 in~\cite{fb}.

Without loss of generality, we may assume that $0 \in \partial \Omega$.

Moreover, $u$ is an energy minimizer for the functional
\[
\int \frac{1}{2} |\nabla u|^2 + W(u) \, \chi_{\{u \ge 0\}} \, dx.
\]

We consider the scaled solutions of \eqref{serrin}, $u_\varepsilon$. defined as \eqref{scaled}, solving
\begin{equation}\label{serrinepsilon}
\begin{cases}
\varepsilon\Delta u = \frac{1}{\varepsilon}W'(u) & \text{in } \Omega_\varepsilon, \\[2mm]
u_\varepsilon > 0, \quad \partial_n u_\varepsilon > 0 & \text{in } \Omega_\varepsilon,\\[1mm]
u_\varepsilon = 0 & \text{on } \partial \Omega_\varepsilon,\\[1mm]
|\nabla u_\varepsilon| = \text{const} & \text{on } \partial \Omega_\varepsilon.
\end{cases}
\end{equation}
Where $\Omega_\varepsilon$ is the epigraph of the blow-down sequence 
\[
\varphi_\varepsilon(x') := \varphi\Big(\frac{x'}{\varepsilon}\Big),
\]
From~\cite{ww}, we know that
\[
u_\varepsilon \to \chi_E \quad \text{in } L^1_{\rm loc},
\]
where $E \subset \mathbb{R}^n$ is a set of minimal perimeter, which can be proved to be a minimal cone. To prove this it is sufficient to reason as in Proposition 11.2 in ~\cite{w}.\\
From Proposition 2.4 and Lemma 3.15 in ~\cite{ww} the monotonicity formula and the energy bound remain valid in this setting, therefore the limit set $E$ is a minimal cone.

As before we consider the blow-down limit (up to a subsequence) of
$\varphi_\varepsilon(x') := \varphi\Big(\frac{x'}{\varepsilon}\Big),$ 
whose subgraphs converge to the minimal perimeter set $E$, which is a minimal cone.

We can now repeat the argument used in the proof of Theorem~\ref{n-7derivatives} to prove that $E$ is an half space.
Repeating the reasoning adopted in Theorem~2.1 of ~\cite{fb}, we prove that $\Omega$ is an half space and $u$ is the one dimensional heteroclinic solution.
\end{proof}
\section{Proof of Theorem \ref{n-2}}
We state the following simple proposition, which also appears in~\cite{ww}, Lemma~5.1.

\begin{proposition}\label{exponential estimates}
Let $u$ be a bounded entire solution of \eqref{dwpde}, where $W$ is a double-well potential as in Definition~\ref{dw} with $p=2$.
Let $\{u=\kappa\}$ be a level set and define the function $\mathrm{dist}:\mathbb{R}^n \to \mathbb{R}$ as the distance from $\{u=\kappa\}$.
Then
\[
\lim_{\mathrm{dist}(x)\to\infty} |u(x)| = 1.
\]
\end{proposition}
We shall make use of Theorem~1.1 in~\cite{FV}.
We now begin with the proof of Theorem~\ref{n-2} and then model the proofs of the remaining propositions on this argument.
   \begin{proof}[Proof of Theorem~\ref{n-2}]
We show that both $\overline{u}$ and $\underline{u}$, which are well defined since $u$ is monotone along $e_n$, are $2$-dimensional solutions of $\Delta v = W'(v)$, and then apply Theorem~1.1 in~\cite{FV} to deduce that the solution is energy minimizing.
We prove the statement for $\overline{u}$; the argument for $\underline{u}$ is analogous.\\
Assume that there exists a sequence $x_h=(x'_h,x^n_h)\in\{u=\kappa\}$ such that $x_h^n\to +\infty$ (respectively $-\infty$) and $x'_h\to\xi'\in\mathbb{R}^{n-1}$.
In all remaining cases, by Proposition~\ref{exponential estimates}, we have $\overline{u}\equiv 1$ (respectively $\underline{u}\equiv -1$).\\
For such a sequence, define
\[
u_h(x',x_n):=u(x',x_n+x_h^n),
\]
which solves $\Delta u_h=W'(u_h)$ and converges in $C^1_{\mathrm{loc}}$ to $\overline{u}$ by the Arzelà--Ascoli theorem.\\
The limit function $\overline{u}=\lim_h u_h$ is an $(n-1)$-dimensional solution of $\Delta v=W'(v)$, which can also be regarded as a function on $\mathbb{R}^n$.\\
By standard elliptic estimates (see~\cite{FV}), the gradient of $u_h$ is uniformly bounded. Hence, there exists $R>0$ such that for all $h$ large enough we have that:
\[
\kappa-\frac{\gamma}{2}<u_h<\kappa +\frac{\gamma}{2}
\quad \text{in } B_R:=B((\xi',0),R),
\]
Viewing $\overline{u}$ as a function on $\mathbb{R}^n$, we have $\partial_n \overline{u}\equiv 0$.
By local $C^1$ convergence, $\partial_n u_h\to 0$ uniformly in $B_R$.
We know that $e_n$ is a convex combination of vectors in the cones $C(\delta)$ for all $\delta : \kappa - \gamma <\delta < \kappa + \gamma$.
As a simple consequence we deduce that $\partial_v u_h(x)\ge 0$ for all $x \in B_R$ such that $u_h(x)=\delta$ and for all $v\in C(\delta)$.\\
We also have that, from $C^1$ uniform convergence, for all $x \in B_R$ , $|\overline{u}(x)| < \delta$ and therefore there exists an $n-2$ dimensional cone $C(x)$ such that $\partial_v\overline{u}(x)=0$ for all $v \in C(x)$.\\
This implies that 
\[
\partial_i \overline{u}\equiv 0 \quad \text{in } B((\xi',0),R)
\quad \text{for all } i=3,\dots,n.
\]
Since $\overline{u}$ is $(n-1)$-dimensional and satisfies $\partial_n \overline{u}\equiv 0$, we may work in $B'_R:=B'(\xi',R)\subset\mathbb{R}^{n-1}$.
Thus, $\overline{u}$ depends only on $x_1,x_2$ in $B'_R$.\\
We now apply the unique continuation principle.
Since $B'_R$ is open, there exists $t>0$ and a ball $B''\subset B'_R$ such that, for every $i=3,\dots,n-1$, we have $B''+t e_i\subset B'_R$.
Define
\[
\overline{u}_i^t(x'):=\overline{u}(x'+t e_i).
\]
Then $\overline{u}-\overline{u}_i^t$ satisfies
\begin{equation}\label{formalip}
\Delta(\overline{u}-\overline{u}_i^t)
= c(x')(\overline{u}-\overline{u}_i^t),
\quad c\in L^\infty(\mathbb{R}^{n-1}),
\end{equation}
by the Lipschitz continuity of $W'$.\\
Since $\overline{u}-\overline{u}_i^t\equiv 0$ in $B''$, the unique continuation principle applied to~\eqref{formalip} yields $\overline{u}-\overline{u}_i^t\equiv 0$ in $\mathbb{R}^{n-1}$.
Repeating this argument for all $i=3,\dots,n-1$, we conclude that $\overline{u}$ depends globally only on $x_1$ and $x_2$, hence it is a $2$-dimensional solution.\\
Therefore, both $\overline{u}$ and $\underline{u}$ are $2$-dimensional, and by Theorem~1.1 of~\cite{FV} we obtain $\overline{u}\equiv 1$ and $\underline{u}\equiv -1$.
This implies the local minimality of $u$, since $u$ is monotone in the $n$-direction with limits $\pm 1$.
Finally, since there exist $n-2$ linearly independent directions of monotonicity on $\{u=\kappa\}$, we apply Theorem~\ref{n-7derivatives} to conclude that $u$ is the one-dimensional heteroclinic solution.
\end{proof}
We can now prove Proposition~\ref{n-2dirp}.
Here we assume monotonicity on $n-2$ cones everywhere, since the unique continuation principle cannot be applied to the general $p$-Laplace operator.

\begin{proof}[Proof of Proposition~\ref{n-2dirp}]
In this case, we observe that $\overline{u}$ is an $(n-1)$-dimensional solution of
\[
\Delta_p v = W'(v),
\]
and that $\partial_n \overline{u} \equiv 0$.
We may have either $\overline{u} \equiv 1$ or $-1 < \overline{u} < 1$. This is a simple consequence of Lemma~2.1 of ~\cite{FV}.\\
It is therefore enough to consider the latter case.\\
We claim that for every open set $B' \subset \mathbb{R}^{n-1}$ we have
\[
\partial_i \overline{u} \equiv 0 \quad \text{in } B' \quad \text{for all} \quad i \in \{3,\dots n\}
\]
Indeed, for each such $B'$ there exist
\[
-1 < \min_{B'} \overline{u} < \max_{B'} \overline{u} < 1,
\]

Arguing as in the proof of the previous theorem ,we define the sequence $u_h(x)=u(x+he_n)$ which converges in $C^1_{loc}$ to $\overline{u}$.
The monotonicity along $n-2$ dimensional cones yields that:
\[
\partial_i \overline{u} \equiv 0 \quad \text{in } B' \quad \text{for all} \quad i \in \{3,\dots n\}
\]
for every $B' \subset \mathbb{R}^{n-1}$.
Therefore, $\overline{u}$ depends only on two variables and is thus a $2$-dimensional solution.\\
The same reasoning applies to $\underline{u}$, and consequently, by Theorem 1.1 of ~\cite{FV}, we obtain
\[
\overline{u} \equiv 1
\quad \text{and} \quad
\underline{u} \equiv -1.
\]
Finally, as in the previous theorem, we apply Theorem~\ref{n-7derivatives} to deduce the one-dimensional symmetry of the solution.
\end{proof}

We can now prove the general statement of Proposition \ref{n-2dir}.

\begin{proof}[Proof of Proposition~\ref{n-2dir}]
We consider $\overline{u}$ in $B'_r$ and observe that $\overline{u}$ depends only on two variables in $B'_r$.
Using the same argument based on the unique continuation principle as in the proof of Theorem~\ref{n-2}, we deduce that both $\overline{u}$ and $\underline{u}$ are $2$-dimensional solutions.\\
In the case where $W \in C^{1,1}_{\mathrm{loc}}$ and $W$ is as in Definition~\ref{dw}, we can directly repeat the argument of Theorem \ref{n-2dir}, obtaining
\[
\overline{u} \equiv 1
\quad \text{and} \quad
\underline{u} \equiv -1,
\]
and hence concluding that $u$ is energy minimizing.\\
In the remaining case, corresponding to a potential as in Definition~\ref{DW}, we note that $\overline{u}$ (and similarly $\underline{u}$) depends only on two variables and is stable .
Therefore, $\overline{u}$ is the one-dimensional heteroclinic solution (see Lemma~3.2 in~\cite{ac}).
Applying Proposition~2.10 of ~\cite{FV}, we conclude that $u$ is energy minimizing.
\end{proof}
\begin{remark}
    We consider $u \in C^1(\mathbb{R}^n)$ to be a solution of \eqref{dwpde}, and $W$ is as in definition \ref{dw} with $p=2$ or as in definition \ref{DW}.\\
    If $\partial_iu>0$ on $\{u=0\}$ for all $i \in {8,\dots n}$, and $u$ satisfies the hypothesis of Proposition \ref{n-2dir}, the $u$ is the one dimensional heteroclinic solution.\\
    This is, to my knowledge, the most restrictive hypothesis on the monotonicity of a solution forcing it to be one dimensional symmetric.
\end{remark}
We now state a proposition that is clearer and represents the precise codimensional analogue of the Ambrosio--Cabrè theorem.
\begin{proposition}\label{n-2dirfin}
Let $u \in C^1(\mathbb{R}^n)$ be a bounded solution of \eqref{dwplap}, where $W$ is a double-well potential as in Definition~\ref{dw} with $p \in (1,+\infty)$, or as in Definition~\ref{DW} with $p=2$.
If
\[
\partial_i u > 0 \quad \text{for all } i = 3,\dots,n,
\]
then $u$ is the one-dimensional heteroclinic solution.
\end{proposition}
\section{Possible developments}
It would be interesting to investigate whether Theorem \ref{n-2} remains valid under the weaker assumption of monotonicity on an $(n-2)$-dimensional cone restricted to the nodal level set ${u=\kappa}$. Another possible direction would be to develop a unique continuation principle for (\ref{dwplap}) in this setting, in order to deduce one-dimensional symmetry from monotonicity in $(n-2)$-dimensional cones defined on smaller domains
\section{Further observations}
For general solutions, the conjecture is false. 
We start by recalling a gradient bound for bounded and entire solutions (see~\cite{w}).

Let $u \in C^2(\mathbb{R}^n)$ be a bounded entire solution of~\eqref{dwpde}.  
Then, for all $x \in \mathbb{R}^n$, it holds
\begin{equation}\label{modicainequality}
\frac{1}{2} |\nabla u(x)|^2 \le W(u(x))
\quad \text{for all } x \in \mathbb{R}^n.
\end{equation}

Under a suitable assumption, we can prove an asymptotic property using exactly the technique employed in~\cite{ww}, Lemma~5.3, which is known as the sliding method and was developed in~\cite{bcn}.
\begin{proposition}\label{sequenzainfinito1d}
Let $u \in C^2$ be a bounded entire solution of~\eqref{dwpde}, where $W$ is a double-well potential as in Definition~\ref{dw} with $\kappa=0$ or $\kappa<0$, such that
\[
\{u \le 0\} \subseteq \{x : x_n \le \varphi(x')\}
\]
for some $\varphi : \mathbb{R}^{n-1} \to \mathbb{R}$.
Then there exists a sequence $x_h$ with $|x_h| \to \infty$ such that the translates
\[
u_h(x) = u(x + x_h)
\]
converge locally in $C^2$ to the one-dimensional heteroclinic solution.
\end{proposition}

\begin{proof}
From Proposition~\ref{exponential estimates}, we have that for all $x' \in \mathbb{R}^{n-1}$,
\[
\lim_{x_n \to +\infty} u(x',x_n) = 1.
\]
Let us denote $\Omega = \{u > 0\}$.\\
The argument is exactly the same as in the proof of Lemma~5.3 in~\cite{ww}.\\
Take $R>0$, arbitrarily large.\\
We take the ball $B_R$ and define the unique radial solution of:
     $$\begin{cases}
        \Delta v_R= W'(v_R)  \quad \text{in} \quad B_R\\
        v_R >0  \quad \text{in} \quad B_R\\
        v_R\equiv 0 \quad \text{on} \quad \partial B_R
    \end{cases}$$
For any $x$ and $R>0$, denote
\[
v_x^R := v^R(\cdot - x).
\]
Since $\sup_{B_R} v^R < 1$, if $t$ is sufficiently large,
\[
v^R_{t e_n} < u \quad \text{in } B_R(t e_n).
\]
Let
\[
t^* := \inf\{t : B_R(t e_n) \subset \Omega\}.
\]
This value is well defined since $\partial \Omega$ lies below the complete graph of $\varphi$, and $B_R(t^* e_n)$ is tangent to $\partial \Omega$ at some point $x_R$.\\
By Lemma~3.1 in~\cite{bcn}, for all $t \ge t^*$,
\[
u > v^R_{t e_n} \quad \text{in } B_R(t e_n).
\]
The Hopf lemma then implies
\[
|\nabla u(x_R)| = \frac{\partial u}{\partial \nu}(x_R)
\ge \frac{\partial v^R_{t^* e_n}}{\partial \nu}(x_R),
\]
where $\nu$ is the upward unit normal to $\partial \Omega$.
Since $B_R(t^* e_n)$ is tangent to $\partial \Omega$ at $x_R$, we have
\[
\frac{\partial v^R_{t^* e_n}}{\partial \nu}(x_R)
= \frac{\partial v^R}{\partial r}(R e_n)
= |\nabla v(R e_n)|.
\]
On the other hand, as $R \to +\infty$, $v^R(R e_n + \cdot)$ converges to a positive solution of
\[
\begin{cases}
\Delta v^\infty = W'(v^\infty), & \text{in } \mathbb{R}^n_+,\\
v^\infty > 0, & \text{in } \mathbb{R}^n_+,\\
v^\infty = 0, & \text{on } \partial \mathbb{R}^n_+.
\end{cases}
\]
Because $v^R$ is radial, $v^\infty$ depends only on the $x_n$ variable and is the one-dimensional heteroclinic solution on $\mathbb{R}^n_+$ (see~\cite{bcn}).\\
In particular,
\[
\sqrt{2W(0)} = |\nabla v^\infty(0)|
= \lim_{R \to +\infty} |\nabla v^R(R e_n)|.
\]
Using the Modica inequality, (\ref{modicainequality}), we deduce that
\[
|\nabla u(x_R)| \ge \frac{\partial v^R_{t^* e_n}}{\partial \nu}(x_R)
\longrightarrow \sqrt{2W(0)}.
\]
We can therefore define the sequence
$u_R(x) = u(x + x_R)$ 
which solves $\Delta w = W'(w)$.
By the Ascoli--Arzelà theorem, $u_R$ converges (up to subsequences) to a limit solution $u^\infty$ such that
\[
\Delta u^\infty = W'(u^\infty)
\quad \text{and} \quad
u^\infty(0) = \sqrt{2W(0)}.
\]
From Theorem~5.1 in~\cite{cgs}, we deduce that $u^\infty$ is the one-dimensional heteroclinic solution.
\end{proof}

\textsc{Gabriele Ferla:}\\
\noindent
\textsc{Dipartimento di Matematica, Universit\`a di Trento,
Via Sommarive, 14, 38123 Povo TN, Italy
}
\\
\noindent
{\em Electronic mail address:}
ferlagabriele@gmail.com
\end{document}